\documentclass[11pt,reqno]{amsart}
\usepackage{amsfonts,amssymb,amscd,amsmath,enumerate,verbatim,calc}
\usepackage{mathtools}
\usepackage{mathabx}
\usepackage{mathrsfs}
\usepackage{verbatim}  
\usepackage{wrapfig}
\usepackage{xcolor}
\usepackage{hyperref}

\newtheorem{thm}{Theorem}[section]
\newtheorem{lem}[thm]{Lemma}

\newtheorem{cor}[thm]{Corollary}
\newtheorem{Q}[thm]{Question}

\theoremstyle{remark}
\newtheorem*{rem}{Remark}
\newtheorem*{ack}{Acknowledgements}

\newcommand{\es}{\emptyset}

\newcommand{\A}{\mathcal{A}}

\newcommand{\la}{\lambda}

\newcommand{\mb}{\mathbb}

\newcommand{\iy}{\infty}

\newcommand{\beqa}{\begin{eqnarray*}}
	\newcommand{\eeqa}{\end{eqnarray*}}

\newcounter{cnt1}
\newcounter{cnt2}
\newcounter{cnt3}
\newcounter{cnt4}
\newcommand{\blr}{\begin{list}{$($\roman{cnt1}$)$} {\usecounter{cnt1}
			\setlength{\topsep}{0pt} \setlength{\itemsep}{0pt}}}
	\newcommand{\blR}{\begin{list}{\Roman{cnt4}.\ } {\usecounter{cnt4}
				\setlength{\topsep}{0pt} \setlength{\itemsep}{0pt}}}
		\newcommand{\bla}{\begin{list}{$(\alph{cnt2})$} {\usecounter{cnt2}
					\setlength{\topsep}{0pt} \setlength{\itemsep}{0pt}}}
			\newcommand{\bln}{\begin{list}{$($\arabic{cnt3}$)$} {\usecounter{cnt3}
						\setlength{\topsep}{0pt} \setlength{\itemsep}{0pt}}}
				\newcommand{\el}{\end{list}}

\begin{document}
				
				
				\baselineskip=17pt
				
				
			\title[On Kakutani's characterization]{On Kakutani's characterization of the closed linear sublattices of $C(X)$ -- Revisited}
				\author[T. Thomas]{Teena Thomas}
				\address{Indian Institute of Technology Hyderabad\\ Sangareddy\\Telangana, India-502284}
				\email{ma19resch11003@iith.ac.in; thomasteena796@gmail.com}

			\begin{abstract}
				In his paper [Concrete representation of abstract {$(M)$}-spaces. ({A} characterization of the space of continuous functions.), \textit{Ann. of Math.}, {\bf 42 (2)} (1941), 994--1024.], S. Kakutani gave an interesting representation of the closed linear sublattices of the space of real-valued continuous functions on a compact Hausdorff space, which is determined by a set of algebraic relations. In this short note, we present a simple proof of this representation without using any profound lattice theory or functional analysis results, making this proof accessible even to undergraduate students. 
				
			\end{abstract}
			 	
			 	
			 	\maketitle

				\section{Preliminaries}\label{S1}
				Let $X$ be a compact Hausdorff space. We denote the Banach space (in other words, a complete normed linear space) of all real-valued continuous functions on $X$, equipped with the supremum norm, by $C(X)$. It is well-known that $C(X)$ is a lattice under the operation of pointwise maximum or minimum of a pair of functions in $C(X)$ and an algebra under the operation of pointwise multiplication of a pair of functions in $C(X)$. All the subspaces in this note are assumed to be closed with respect to the norm topology. We say that a closed linear subspace $\mathcal{A}$ of $C(X)$ is a {\it sublattice} of $C(X)$ if $\mathcal{A}$ is a lattice in its own right under the same lattice operations as in $C(X)$. A closed linear subspace $\mathcal{A}$ of $C(X)$ is said to be a {\it subalgebra} of $C(X)$ if $\mathcal{A}$ is an algebra in its own right under the same algebra operations as in $C(X)$. Kakutani presented an algebraic characterization of the closed linear sublattices of $C(X)$ as follows$\colon$
				
				\begin{thm}[{\cite[Theorem~3, pg.~1005]{Ka}}]\label{T2.0}
					Let $X$ be a compact Hausdorff space. Let $\mathcal{A}$ be a closed linear subspace of $C(X)$. Then $\mathcal{A}$ is a sublattice of $C(X)$ if and only if there exists an index set $I$ and co-ordinates $(t_i,s_i,\lambda_i) \in X \times X\times [0,1]$, for each $i \in I$ such that 
					\begin{equation}\label{Eqn1.0}
						\mathcal{A} = \{f \in C(X)\colon f(t_i)= \lambda_i f(s_i)\mbox{, for each }i \in I\}.
					\end{equation}	
				\end{thm}
				
				In Section~\ref{Sec3}, we provide an elementary proof of the result above. In fact, we exploit the structure of the closed linear sublattices of the two dimensional vector space, $\mb{R}^{2}$, in this proof. Moreover, in this proof, we apply a few auxiliary results whose proofs assume the knowledge of a basic course in real analysis and point-set topology.
				
				Let us first understand what the sublattices look like in $\mb{R}^{2}$ under the lattice operation of co-ordinate-wise maximum or minimum of a pair of vectors in $\mb{R}^{2}$. For simplicity, we denote the linear span of a vector $(a,b) \in \mb{R}^{2}$ as $span\{(a,b)\}$. Let $\A$ be a closed linear subspace of $\mb{R}^{2}$. Basic knowledge of linear algebra tells us that the possible dimensions of $\mathcal{A}$ are $0,1$ or $2$. If the dimension of $\mathcal{A}$ is either $0$ or $2$, then $\mathcal{A} = \{(0,0)\}$ or $\mathcal{A} = \mathbb{R}^{2}$ respectively; obviously, these subspaces are sublattices of $\mb{R}^{2}$. Assume that $\mathcal{A}$ is a one dimensional subspace of $\mathbb{R}^{2}$. We know that if $\A = span\{(a,b)\}$ for some $(a,b) \in \mb{R}^{2}$, then $\A = span\{(\la a,\la b)\}$ for each $\la \in \mb{R}\backslash \{0\}$. Therefore, without loss of generality, there exists $a,b \in [-1,1]$ such that $\A = span\{(a,b)\}$. It is easy to see that if $-1\leq a,b \leq 0$ or $0\leq a,b\leq 1$ then $\A$ is a sublattice of $\mb{R}^{2}$. Furthermore, if $-1 \leq a < 0 < b \leq 1$ then the minimum of $(a,b)$ and $(2a,2b)$ is $(2a,b) \not\in \mathcal{A}$; hence $\mathcal{A}$ is not a sublattice of $\mb{R}^{2}$. Using a similar argument, if $-1 \leq b < 0 < a \leq 1$ then $\A$ is not a sublattice of $\mb{R}^{2}$. Therefore, without loss of generality, we arrive at the following conclusion. 
				
				\begin{lem}\label{L1.1}
					Consider $\mathbb{R}^{2}$ as a lattice under the operation of co-ordinate-wise maximum or minimum of a pair of vectors in $\mb{R}^{2}$. Then the only closed linear sublattices of $\mathbb{R}^{2}$ are $\{(0,0)\}$, $\mathbb{R}^{2}$ and $span\{(a,b)\}$, for those $(a,b) \in \mathbb{R}^{2}$ satisfying $0\leq a,b \leq 1$.
				\end{lem}

				We can also list all possible closed linear subalgebras of $\mathbb{R}^{2}$ as follows$\colon$
				\begin{lem}[{\cite[Lemma~4.46]{Fo}}]\label{L1}
					Consider $\mathbb{R}^{2}$ as an algebra under co-ordinate-wise addition and multiplication of a pair of vectors in $\mb{R}^{2}$. Then the only subalgebras of $\mathbb{R}^{2}$ are $\{(0,0)\}$, $\mathbb{R}^{2}$, $span\{(1,0)\}$, $span\{(0,1)\}$ and $span\{(1,1)\}$.
				\end{lem}

				We now recall a few known facts and definitions needed to prove Theorem~\ref{T2.0}. For a compact Hausdorff space $X$, we say a closed linear subspace $\A$ of $C(X)$ {\it separates the points of $X$} if for every two distinct points $x,y \in X$, there exists $f \in \mathcal{A}$ such that $f(x) \neq f(y)$. 
				
				The following result shows the interconnection between a subalgebra and a sublattice of $C(X)$.
				
				\begin{lem}[{\cite[Lemma~4.48]{Fo}}]\label{L0}
					Let $X$ be a compact Hausdorff space. If $\mathcal{A}$ is a closed linear subalgebra of $C(X)$, then $\mathcal{A}$ is a sublattice of $C(X)$.
				\end{lem}
				
				Furthermore, for a compact Hausdorff space $X$ and a closed linear sublattice $\mathcal{A}$ of $C(X)$, a sufficient condition for a function in $C(X)$ to be in $\mathcal{A}$ is
				\begin{lem}[{\cite[Lemma~4.49]{Fo}}]\label{L2}
					Let $X$ be a compact Hausdorff space. Let $\mathcal{A}$ be a closed linear sublattice of $C(X)$ and $f \in C(X)$. If for every $x,y \in X$ there exists $g_{xy} \in \mathcal{A}$ such that $g_{xy}(x) = f(x)$ and $g_{xy}(y) =f(y)$ then $f \in \mathcal{A}$.
				\end{lem}

				Let $T$ be a locally compact Hausdorff space. We denote the space of real-valued continuous functions on $T$ vanishing at infinity by $C_{0}(T)$. Further, let us denote $T_{\infty}$ as the one-point compactification of $T$ and $t_{\infty}$ be the \textquotedblleft point at infinity\textquotedblright. Consider the restriction map $\Phi \colon J_{t_\infty} \rightarrow C_0(T)$ defined as $\Phi(f) = f \vert_{T}$, for $f \in J_{t_\infty}$. Then the map $\Phi$ is an isometric lattice isomorphism (in other words, a map which preserves the distances as well as the lattice and linear structures). The identification above and Theorem~\ref{T2.0} help us to algebraically characterize the closed linear sublattices and subalgebras of the space $C_0(T)$ in Section~\ref{Sec3}. 
				 \section{Main Results}\label{Sec3}
				 Let us now prove our main result. 
				
				 \begin{proof}[Proof of Theorem~\ref{T2.0}]
				 	We assume that $X$ contains at least two points because if $X$ is a singleton set then $C(X)$ is simply $\mb{R}$ and the only closed linear sublattices or subalgebras are $\{0\}$ and $\mb{R}$. If $\mathcal{A}$ has the description as in (\ref{Eqn1.0}) then clearly $\mathcal{A}$ is a sublattice of $C(X)$.
				 	
				 	Now, assume that $\mathcal{A}$ is a sublattice of $C(X)$.
				 	For every two distinct points $x,y \in X$, define \[\mathcal{A}_{xy} = \{(g(x),g(y)) \in \mathbb{R}^{2}\colon g \in \mathcal{A}\}.\] Since $\mathcal{A}$ is a lattice, $\mathcal{A}_{xy}$ is a sublattice of $\mathbb{R}^{2}$ (under the operation of co-ordinate-wise maximum of a pair of vectors in $\mb{R}^{2}$). 
				 	
				 	{\sc Case 1$\colon$ }Assume $\mathcal{A}$ separates the points of $X$. 
				 	
				 	If for every $x,y \in X$, $\mathcal{A}_{xy} = \mathbb{R}^{2}$ then by Lemma~\ref{L2}, $\mathcal{A} = C(X)$. Hence $\mathcal{A}$ has the the description as in (\ref{Eqn1.0}). Otherwise, there exists two distinct points $x_0,y_0 \in X$ such that $\mathcal{A}_{x_{0}y_{0}}$ is a proper sublattice of $\mathbb{R}^{2}$. 
				 	Consider the following collection$\colon$ \[I = \{(t,s,\lambda)\in X\times X \times [0,1]\colon f(t) = \lambda f(s)\mbox{, for each }f \in \mathcal{A}\}.\] 
				 	
				 	We now show that $I \neq \emptyset$. Since $\mathcal{A}$ separates the points of $X$, $\mathcal{A}_{x_{0}y_{0}}$ cannot be $\{(0,0)\}$ or $span\{(1,1)\}$. Thus $\mathcal{A}_{x_{0}y_{0}} = span \{(a,b)\}$ for some $0 \leq a,b <1$ and $a \neq b$. 
				 	If $a=0$ and $b>0$ then for each $g \in \mathcal{A}$, $g(x_0)=0$ and hence $(x_0,y_0,0) \in I$. If $a>0$ and $b=0$ then for each $g \in \mathcal{A}$, $g(y_0)=0$ and hence $(y_0,x_0,0) \in I$. If without loss of generality $0<a<b$ then for each $g \in \mathcal{A}$, there exists $r_{g}\in  \mathbb{R}$ such that $(g(x_0), g(y_0)) =  (r_{g}a,r_{g}b)$. It follows that $g(x_0) =\frac{a}{b} g(y_0)$. Thus $(x_0,y_0,\frac{a}{b}) \in I$.
				 	
				 	We index $I$ by $I$ itself, that is, each element of $I$ is indexed by itself. Further, define \[\mathcal{A}^{\prime} = \{f \in C(X)\colon f(t_i) = \lambda_i f(s_i)\mbox{, for each } i\in I\}.\] By the definition of $I$, it is clear that $\mathcal{A} \subseteq \mathcal{A}^{\prime}$.
				 	
				 	We next show that $\mathcal{A}^{\prime} \subseteq \mathcal{A}$. Let $f \in \mathcal{A}^{\prime}$. In order to show $f \in \mathcal{A}$, by Lemma~\ref{L2}, it suffices to show that for each $x,y \in X$, there exists $g_{xy} \in \mathcal{A}$ such that $g_{xy}(x) = f(x)$ and $g_{xy} (y) = f(y)$. Therefore, it suffices to show that for each $x,y \in X$, $(f(x),f(y)) \in \mathcal{A}_{xy}$.
				 	
				 	Let $x, y \in X$. Since $\mathcal{A}$ separates the points of $X$, $\mathcal{A}_{xy}$ cannot be $\{(0,0)\}$ or $span\{(1,1)\}$. We remark here that in this case, we use the assumption that $\A$ separates the points of $X$ only to prove that $I \neq \es$ and to rule out the above two possibilities of $\mathcal{A}_{xy}$. 
				 	
				 	If $\mathcal{A}_{xy} = \mathbb{R}^{2}$ then clearly $(f(x),f(y)) \in \mathbb{R}^{2} = \mathcal{A}_{xy}$. 
				 	If $\mathcal{A}_{xy} = span \{(0,1)\}$ then $(0,f(y)) \in \mathcal{A}_{xy}$ and for each $g \in \mathcal{A}$, $g(x) = 0$. Thus $(x,y,0) \in I$. Since $f \in \mathcal{A}^{\prime}$, $f(x) = 0$. Hence $(f(x), f(y)) \in \mathcal{A}_{xy}$. Similar arguments hold if $\mathcal{A}_{xy} = span\{(1,0)\}$.
				 	
				 	Without loss of generality, let $0< a<b <1$. Consider $\mathcal{A}_{xy} = span\{(a,b)\}$ then for each $g \in \mathcal{A}$, $g(x) = \frac{a}{b} g(y)$. Thus $(x,y,\frac{a}{b}) \in I$. Since $f \in \mathcal{A}^{\prime}$, let $\frac{f(x)}{a} = \frac{f(y)}{b} = r$ (say). It follows that $(f(x),f(y)) \in span\{(a,b)\} = \mathcal{A}_{xy}$.          	   
				 	
				 	{\sc Case 2$\colon$ }Assume that $\mathcal{A}$ does not separate the points of $X$. Thus there exists two distinct points $x_0,y_0 \in X$ such that $f(x_0) = f(y_0)$, for each $f \in \mathcal{A}$. Consider the following collection$\colon$ \[I = \{(t,s,\lambda)\in X\times X \times [0,1]\colon f(t) = \lambda f(s)\mbox{, for each }f \in \mathcal{A}\}.\] Since $(x_0,y_0,1), (y_0,x_0,1) \in I$, clearly $I \neq \emptyset$. We index $I$ by $I$ itself, that is, each element of $I$ is indexed by itself. Let us define \[\mathcal{A}^{\prime} = \{f \in C(X)\colon f(t_i) = \lambda_i f(s_i)\mbox{, for each } i\in I\}.\] Clearly $\mathcal{A} \subseteq \mathcal{A}^{\prime}$. In order to show $\mathcal{A}^{\prime} \subseteq \mathcal{A}$, by Lemma~\ref{L2} it suffices to show that for each $f \in \mathcal{A}^{\prime}$ and each $x,y \in X$, $(f(x),f(y)) \in \mathcal{A}_{xy}$. 
				 	
				 	Let $x,y \in X$. We first consider the following possibilities of $\mathcal{A}_{xy}\colon$ $\mb{R}^{2}$, $span\{(0,1)\}$, $span\{(1,0)\}$ and $span\{(a,b)\}$ for those $a,b \in \mb{R}$ satisfying $0< a < b<1$ (without loss of generality). For each of the above possibilities, we apply arguments similar to that used in {\sc Case 1} to show that $(f(x),f(y)) \in \mathcal{A}_{xy}$. 
				 	
				 	In this case, due to our assumption that $\A$ does not separate the points of $X$, we need to consider the two possibilities which we rule out in {\sc Case 1}. They are as follows$\colon$
				 	If $\mathcal{A}_{xy} = \{(0,0)\}$ then for each $g \in \mathcal{A}$, $g(x) = 0= g(y)$. Hence, $(x,y,0), (y,x,0) \in I$. Since $f \in \mathcal{A}^{\prime}$, $(f(x),f(y)) = (0,0) \in \mathcal{A}_{xy}$. 
				 	If $\mathcal{A}_{xy} = span\{(1,1)\}$ then for each $g \in \mathcal{A}$, $g(x) = g(y)$. Thus $(x,y,1) \in I$. Since $f \in \mathcal{A}^{\prime}$, $(f(x),f(y)) = (f(x),f(x)) \in \mathcal{A}_{xy}$.
				 \end{proof}
				 
				 The closed linear subalgebras of the $C(X)$ space has a representation similar to that in (\ref{Eqn1.0}) with the value of the coefficients $\lambda_{i}$ being either $0$ or $1$. With the help of Lemmas~\ref{L1}, \ref{L0} and \ref{L2}, the following result is proved using a similar argument as in Theorem~\ref{T2.0} and hence we omit it.
				 \begin{thm}\label{T2}
				 	Let $X$ be a compact Hausdorff space. Let $\mathcal{A}$ be a closed linear subspace of $C(X)$. Then $\mathcal{A}$ is a subalgebra of $C(X)$ if and only if there exists an index set $I$ and co-ordinates $(t_i,s_i,\lambda_i) \in X \times X\times \{0,1\}$, for each $i \in I$ such that 
				 	\begin{equation}\label{Eqn1}
				 		\mathcal{A} = \{f \in C(X)\colon f(t_i)= \lambda_i f(s_i)\mbox{, for each }i \in I\}.
				 	\end{equation}	
				 \end{thm}
				 
				The following characterizations of the sublattices and subalgebras of $C_{0}(T)$ for a locally compact Hausdorff space $T$ follows directly from Theorems~\ref{T2.0} and \ref{T2} and the fact that $C_{0}(T)$ is isometrically lattice isomorphic to the subalgebra $J_{t_{\iy}}$ of $C(T_{\iy})$.
				 
				 \begin{cor}
				 	Let $T$ be a locally compact Hausdorff space. Let $\mathcal{A}$ be a closed linear subspace of $C_{0}(T)$. Then
				 	\blr
				 		\item $\mathcal{A}$ is a sublattice of $C_{0}(T)$ if and only if there exists an index set $I$ and co-ordinates $(t_i,s_i,\lambda_{i}) \in T \times T\times [0,1]$, for each $i \in I$ such that 
				 		\begin{equation}
				 			\mathcal{A} = \{f \in C_{0}(T)\colon f(t_i)= \lambda_i f(s_i)\mbox{, for each }i \in I\}.
				 		\end{equation}	
				 		\item $\mathcal{A}$ is a subalgebra of $C_{0}(T)$ if and only if there exists an index set $I$ and co-ordinates $(t_i,s_i,\lambda_{i}) \in T \times T\times \{0,1\}$, for each $i \in I$ such that 
				 		\begin{equation}
				 			\mathcal{A} = \{f \in C_{0}(T)\colon f(t_i)= \lambda_{i} f(s_i)\mbox{, for each }i \in I\}.
				 		\end{equation}	
				 	\el
				 \end{cor}
				 
				\begin{rem}\label{R1}
					Kakutani provided a precise identification of a much general class of Banach spaces, namely abstract $(M)$-spaces (see \cite[pg.~994]{Ka} for the definition), with a subspace of $C(X)$ for some compact Hausdorff space $X$. Let $X$ be a compact Hausdorff space and $\A$ be a closed subspace of $C(X)$. If $\A$ is an abstract $(M)$-space, then by \cite[Theorem~1, pg.~998]{Ka}, there exists a compact Hausdorff space $\Omega$ such that $\A$ is isometric and lattice isomorphic to the subspace, $\{f \in C(\Omega)\colon f(t_i)= \lambda_i f(s_i)\mbox{, for each }i \in I\}$, of $C(\Omega)$ for some index set $I$ and co-ordinates $(t_i,s_i,\la_i) \in \Omega \times \Omega \times [0,1]$ for $i \in I$. However, it is not necessary that $\A$, being a subspace of $C(X)$, has a description  in $C(X)$ as given in (\ref{Eqn1.0}). For example, consider $\A$ to be the space of real-valued affine continuous functions on $[0,1]$. Now, $\A$ is a closed subspace of $C([0,1])$ but not a sublattice of $C([0,1])$. Therefore, by Theorem~\ref{T2.0}, $\A$ does not have a description as in (\ref{Eqn1.0}), for any given subfamily of co-ordinates in $[0,1] \times [0,1] \times [0,1]$. Nevertheless, it is easy to see that $\A$ is isometric and lattice isomorphic to $C(\{0,1\})$ and hence is an abstract $(M)$-space.
				\end{rem}
			
			    \begin{ack}
			    	The author would like to thank Prof. Dirk Werner for indicating the Remark.
			    \end{ack}

			\end{document}